\documentclass[11pt,a4paper]{article}
\usepackage[margin=26mm]{geometry}
\usepackage[T1]{fontenc}
\usepackage{lmodern,amsmath,amssymb,amsthm,microtype}
\usepackage{booktabs,longtable,array}
\usepackage{xcolor}
\definecolor{linkblue}{RGB}{24,67,111}
\usepackage[colorlinks=true,linkcolor=linkblue,citecolor=linkblue,urlcolor=linkblue,bookmarksnumbered=true]{hyperref}
\hypersetup{pdftitle={Ordinary and measurable colorings of Euclidean three-space},pdfauthor={}}
\newtheorem{theorem}{Theorem}[section]
\newtheorem{lemma}[theorem]{Lemma}
\newtheorem{corollary}[theorem]{Corollary}

\theoremstyle{definition}
\newtheorem{definition}{Definition}[section]
\newtheorem*{standard}{Standard theorem}
\theoremstyle{remark}

\newcommand{\R}{\mathbb R}
\newcommand{\Q}{\mathbb Q}
\newcommand{\Z}{\mathbb Z}
\newcommand{\T}{\mathbb T}
\newcommand{\C}{\mathbb C}
\newcommand{\one}{\mathbf 1}
\newcommand{\ip}[2]{\langle #1,#2\rangle}
\newcommand{\norm}[1]{\lVert #1\rVert}
\newcommand{\vol}{\operatorname{vol}}
\newcommand{\Id}{\operatorname{Id}}
\newcommand{\supp}{\operatorname{supp}}
\title{Ordinary and measurable colorings\\of Euclidean three-space}
\author{Mikhail Koshelev, Andrey Kupavskii, Vsevolod Voronov}
\date{}
\begin{document}
\maketitle
\begin{abstract}
A proper coloring of Euclidean space is a coloring that assigns different colors to points
at distance one. Using the planar character-law rigidity theorem of OpenAI, we prove
that every finite proper coloring of $\R^3$ can be replaced by a proper Lebesgue-measurable
coloring with the same number of colors. As a corollary we obtain that the ordinary and
measurable chromatic numbers of $\R^3$ are equal, and both lie between $7$ and $15$.
\end{abstract}

\section{Introduction}

The unit-distance graph of $\R^d$ has vertex set $\R^d$, with two points adjacent if and only if the
Euclidean distance between them equals one. The chromatic number $\chi(\R^d)$ is the least number
of colors in a proper coloring of this graph. Numerous results have been obtained concerning
both bounds on $\chi(\R^n)$ and various modifications of this problem. The lower bound on this
value for $n$ tending to $\infty$ was subsequently improved in the works of Raiskii~\cite{Raiskii1970}, Larman
and Rogers~\cite{LarmanRogers1972}, Frankl and Wilson~\cite{FranklWilson1981}, Raigorodskii~\cite{Raigorodskii2000}, and Davies~\cite{Davies}. Other studies
have addressed bounds on the chromatic numbers of spaces of small dimension~\cite{BogolyubskyRaigorodskii2019, OpenAI, Davies},
colorings of spheres~\cite{Kupavskii2011spheres, Raigorodskii2012, Prosanov2018}, spaces with other norms~\cite{Kupavskiy2011norm}, real and rational spaces with
several forbidden distances~\cite{RaigorodskiiShitova2008}, and colorings with a forbidden monochromatic equilateral
triangle~\cite{ZvonarevRaigorodskiiSamirovKharlamova2014}.

The measurable chromatic number $\chi_m(\R^d)$ is defined in a similar way with an additional
requirement for every color class to be a measurable set. The inequality
\[
  \chi(\R^d) \le \chi_m(\R^d)
\]
is thus immediate, but the question of whether the inequality can be replaced by equality
turned out to be quite tricky.

This question was first addressed by Sz\'ekely in his work on measurable chromatic
numbers of geometric graphs (see~\cite{Szekely1984}). Equality can fail when the allowed unit displacements
are restricted. For the graph in $\R^3$ that retains only unit edges whose displacement has rational
coordinates, Payne~\cite[Theorem~1 and Proposition~2]{Payne2009} proves that the ordinary chromatic
number is $2$, whereas the measurable chromatic number is at least $6$.

In three dimensions, Nechushtan~\cite{Nechushtan2002} proved the unrestricted lower bound $6$, whereas de
Oliveira Filho and Vallentin~\cite{OFV} proved the measurable lower bound $7$. At the other end, the
geometric constructions of Coulson~\cite{Coulson} and Radoi\v{c}i\'c and T\'oth~\cite{RT} give $15$ colors.

The lower bounds reflect different kinds of obstruction. By the compactness theorem
of de Bruijn and Erd\H{o}s~\cite{dBE}, failure of ordinary $k$-colorability is witnessed by a finite graph.
Measurability permits analytic restrictions on the color classes, from Falconer's planar bound
$\chi_m(\R^2) \ge 5$~\cite{Falconer1981, BachocNebeOliveiraFilhoVallentin2009} to spectral bounds based on continuous positive kernels~\cite{BachocNebeOliveiraFilhoVallentin2009} and operators~\cite{BachocDeCorteOliveiraFilhoVallentin2014}.
The three-dimensional bound of~\cite{OFV} follows from an upper-density bound strictly below $1/6$
for measurable sets avoiding distance one.

The recent planar character-law rigidity theorem of OpenAI~\cite{OpenAI} supplies a way to extract
regularity from a coloring without assuming measurability. Its planar transfer argument passes
through invariant random colorings, continuous factors and measurable samples. Davies~\cite[Theorem~1.1]{Davies}
proves that ordinary and measurable chromatic numbers agree in every dimension
at least four. We adapt the planar machinery to dimension three, using a deformation of the
sphere to obtain continuity in the remaining coordinate.

\begin{theorem}\label{thm:main}
For every positive integer $k$, the existence of a proper $k$-coloring of $\R^3$ implies
the existence of a proper Lebesgue-measurable $k$-coloring. Consequently,
\[
  \chi(\R^3) = \chi_m(\R^3).
\]
\end{theorem}

As an immediate corollary we obtain the following new lower bound on $\chi(\R^3)$.

\begin{corollary}\label{cor:bounds}
\[
  7 \le \chi(\R^3) = \chi_m(\R^3) \le 15.
\]
\end{corollary}

Throughout the paper we work in ZFC and use completed Lebesgue measure. The final
compactness argument extends a coloring from a conull set while respecting all unit edges and
the colors already assigned. Its color classes change only within the null remainder, which
ensures Lebesgue measurability.

\paragraph{Sketch of the proof.}  The first and crucial stage replaces an arbitrary proper coloring by a measurable partition for which monochromatic unit pairs occur on a set of measure zero in position and direction. The second stage removes those exceptional pairs without adding colors, which results in the proper measurable coloring. 

For the first stage, the original spatial color indicators cannot be integrated: no measurability was assumed. Instead restrict the coloring to the countable dense set of algebraic points and consider the space of \emph{all} proper colorings of that set. Averaging translates and rotations of one coloring gives an invariant probability distribution on whole colorings. The random variable $f_i$ asks whether the origin has color $i$. Its translate asks the same question at another point. Properness becomes an identity saying that the expected product of these two indicators is zero at every algebraic unit displacement.

We first retain the observations that change continuously in expected squared error under small horizontal translations. The resulting functions $p_i$ are nonnegative weights summing to one. The frequency description of this operation lets rigidity identify what was discarded, and shows that its removal preserves zero correlations at nonvertical unit displacements. Positivity alone would not give this preservation; it is where rigidity is used.

To handle the height coordinate, start with a nonvertical unit vector and change its height slightly. Compensate by changing the horizontal length so the vector remains on the unit sphere. The compensated displacement still has zero correlation, and horizontal continuity controls the cost of undoing the compensation. This estimate permits averaging over small algebraic height shifts. The resulting weights $a_i$ are continuous in expected squared error under all translations, and their unit correlations remain zero. After choosing representatives measurable in position and in the random coloring, one probability sample yields the required spatial partition.

For the second stage, replace each spatial class by its density-one points: points at which the class occupies a proportion tending to one of every shrinking ball. These sets avoid unit distance and cover almost all of space. For any finite set of vertices, many sufficiently small common translations preserve the colors already prescribed at its density-one points and put all other vertices into the covered set. Translating back gives a compatible coloring of that finite set. Compactness then supplies a global coloring preserving all the prescribed colors. The translation may depend on the finite set; no single translation is required to work everywhere.

The invariant probability model and sampling argument follow the planar transfer method of~\cite[Proposition~4.2]{OpenAI}; the continuous-observation identification has its counterpart in~\cite[Lemma~4.1]{OpenAI}. Our external rigidity input is the statement of~\cite[Theorem~2.3]{OpenAI}, and the algebraic averaging construction follows~\cite[Section~2.2]{OpenAI}. The density and completion arguments are related to~\cite[Lemma~B.1 and proof of Lemma~B.2]{Davies}, with the density-point step credited there to~\cite[proof of Lemma~5.6]{DOV}. 

The main proof introduces each object when it becomes useful. Appendix~\ref{app:probability} gives the probability and function-space background and proves the representative lemma. Appendix~\ref{app:frequencies} proves the identification of continuous observations with continuous frequencies. Appendix~\ref{app:height-averaging} constructs the algebraic-height average used at the sphere step.

\section{The intermediate measurable partition}\label{sec:weak}

Write $B(0,R)$ for the Euclidean ball of radius $R$. For $d\geq2$, let $\sigma_d$ be normalized surface measure on the unit sphere $S^{d-1}$; for $d=1$, give each of the two unit directions $-1,1$ mass $1/2$. Thus choosing $u$ according to $\sigma_d$ means choosing a uniform unit direction. In three-space we abbreviate $\sigma_3$ to $\sigma$.

\begin{definition}
A measurable partition $B_1,\ldots,B_k$ of $\R^d$ is \emph{weakly proper} if, for every $i$ and $R>0$,
\begin{equation}\label{eq:weak}
 \int_{B(0,R)}\int_{S^{d-1}}
 \one_{B_i}(x)\one_{B_i}(x+u)\,d\sigma_d(u)\,dx=0.
\end{equation}
\end{definition}

The integrand is the indicator of a monochromatic unit pair of color $i$ with first endpoint $x$ and direction $u$. Because it is nonnegative, the condition says that such pairs occur on a null set for this position-and-direction measure. It concerns pairs at distance one, rather than almost every pair of arbitrary points in $\R^d\times\R^d$.

Changing a partition on a spatial null set $N$ does not affect this condition. For a fixed direction $u$, only positions in $N\cup(N-u)$ can change the integrand, and that set is null. Integrating over directions gives a product-null exception by Tonelli's theorem. Completed-measurable classes can consequently be replaced by Borel representatives for this calculation; this also makes the product measurability explicit. The needed product-integration statement is recalled in Appendix~\ref{app:representatives}.

\begin{theorem}[Transfer to weak properness]\label{thm:transfer}
Assume Theorem~\ref{thm:rigidity}. If $\R^3$ has a proper $k$-coloring, where $k$ is a positive integer, then it has a weakly proper measurable $k$-partition.
\end{theorem}

We prove this theorem in the next sections. Lemma~\ref{lem:completion} then converts any weakly proper finite partition into a proper measurable coloring with the same number of colors, completing Theorem~\ref{thm:main}.

\section{A probability model for the coloring}\label{sec:model}

Assume that a proper $k$-coloring of $\R^3$ is given. We first
replace it by a probability distribution on colorings of a countable
set of spatial points. The original color classes need not be
measurable for this construction, which follows the planar
model in~\cite[Proposition~4.2]{OpenAI}.

Let $F=\overline{\Q}\cap\R$ be the field of real algebraic numbers,
and write
\[
 E=F+iF=\overline{\Q},\qquad A=E\times F\subset\C\times\R=\R^3.
\]
Thus $(z,t)\in A$ records a horizontal displacement $z$ and a
vertical displacement $t$. The field $F$ is countable and dense in
$\R$. It is also closed under taking positive square roots of
positive elements. Countability will make the space of colorings
manageable; closure under square roots will later keep a deformation
along the unit sphere inside $A$.

Restrict the given coloring to $A$, and let
\[
 \Omega=\{\omega:A\to\{1,\ldots,k\}:
       \omega(x)\ne\omega(y)\text{ whenever }|x-y|=1\}.
\]
A point $\omega$ of $\Omega$ is an entire proper coloring of $A$.
Neighborhoods specify its colors at finitely many points. Since $A$
is countable, a diagonal subsequence argument gives compactness:
from any sequence of colorings, successively fix the color at the
first point, the second point, and so on. The limit remains proper
because each edge constraint concerns only two coordinates. This
topology is metrizable, and $\Omega$ is nonempty by our initial
restriction. We use its Borel events, generated by the open sets,
to construct the probability measure.

Let
\[
 K=\{u\in E:|u|=1\},\qquad R_u(z,t)=(uz,t).
\]
The elements of $K$ are the algebraic rotations of the horizontal
plane. Together with translations by $A$, they form a countable
group $\Gamma$ of motions $x\mapsto a+R_ux$. A motion acts on
colorings by
\[
 (h\omega)(x)=\omega(h^{-1}x).
\]
This action permutes coordinates, is continuous on $\Omega$, and
preserves properness because the motions preserve distances.

The finite-averaging fact proved from standard results in
Appendix~\ref{app:finite-averaging} supplies finite nonempty sets
$\Lambda_n\subset\Gamma$ such that
\[
 \frac{|h\Lambda_n\mathbin{\triangle}\Lambda_n|}{|\Lambda_n|}
 \longrightarrow0\qquad(h\in\Gamma).
\]
Here $\triangle$ denotes symmetric difference. In words, each fixed
motion changes $h$ only a vanishing fraction of the averaging set.
These are called F\o lner sets. Fix one $\omega_0\in\Omega$ and
give equal probability to its finitely many moved copies:
\[
 \mu_n=\frac1{|\Lambda_n|}\sum_{g\in\Lambda_n}\delta_{g\omega_0}.
\]
The measure $\delta_\omega$ is concentrated at the single coloring
$\omega$. Repetitions among the moved colorings cause no problem:
$\mu_n$ is the law obtained by choosing an index $g$ uniformly.

Probabilities on a compact metric space have a weakly convergent
subsequence. Thus, after passing to a subsequence, $\mu_n$ converges
to a Borel probability $\mu$, meaning that the integrals of every
continuous test function converge. For such a function $\psi$ and
every fixed $h\in\Gamma$,
\[
 \left|\int_\Omega\psi(h\omega)\,d\mu_n(\omega)
       -\int_\Omega\psi(\omega)\,d\mu_n(\omega)\right|
 \le\norm{\psi}_\infty
       \frac{|h\Lambda_n\triangle\Lambda_n|}{|\Lambda_n|}.
\]
Passing to the limit proves invariance under $h$, since continuous
functions determine Borel probabilities on a compact metric space.
Consequently $\mu$ is invariant under all algebraic translations
and horizontal rotations. We complete this measure by including
all subsets of probability-zero Borel sets. The result is an
ordinary countably additive probability on entire colorings;
Appendix~\ref{app:probability} records the measure conventions.

\subsection{Color observations and their correlations}

We work with square-integrable complex functions of this random
coloring. They form the Hilbert space $L^2(\Omega,\mu)$, with
\[
 \ip{v}{w}=\int_\Omega\overline{v(\omega)}w(\omega)\,d\mu(\omega),
 \qquad \norm{v}_2^2=\ip{v}{v}.
\]
The inner product is conjugate-linear in its first argument.
Functions equal with probability one are identified. The squared
distance $\norm{v-w}_2^2$ is their expected squared difference;
for indicators, it is the probability that their answers disagree.

For $g\in A$, put $\tau_g\omega(x)=\omega(x-g)$ and
$T_gv=v\circ\tau_{-g}$. This sign convention means that the
translated observation checks a displaced spatial coordinate.
In particular, define $f_i$ by inspecting the origin; then
\[
 f_i(\omega)=\one_{\{\omega(0)=i\}},
 \qquad (T_gf_i)(\omega)=\one_{\{\omega(g)=i\}}.
\]
Measure invariance shows that each $T_g$ preserves norms and inner
products, and
$T_gT_h=T_{g+h}$, with inverse $T_{-g}$. Such operators are called
unitary. We have $0\le f_i\le1$ and $\sum_{i=1}^k f_i=1$.
Moreover, properness of every coloring in $\Omega$ gives
\begin{equation}\label{eq:original-correlation}
 \ip{f_i}{T_{(z,t)}f_i}=0
 \qquad\bigl((z,t)\in A,\ |z|^2+t^2=1\bigr).
\end{equation}
The scalar product is the probability that the two fixed points
$0$ and $(z,t)$ both have color $i$. The coloring varies randomly;
the displacement does not. These functions of colorings, rather
than spatial indicators of the original coloring, are now our
working objects.

\section{Keeping horizontal continuity}\label{sec:horizontal}

An observation $v\in L^2(\mu)$ is horizontally continuous in mean
square if
\[
 \norm{T_{(z,0)}v-v}_2\longrightarrow0
 \qquad\text{as }z\in E\text{ tends to }0\text{ in }\R^2.
\]
Let $H_h$ be the subspace of these observations. This is continuity
of translated observations in expected squared error; it imposes
no pointwise continuity on an individual coloring.

The next lemma identifies the orthogonal projection onto this
subspace as an averaging operation that preserves nonnegativity.
We give its form for a dense translation group, so that it also
applies later to the algebraic vertical shifts. Here the group is
$G=E$ and the operators are $T_{(z,0)}$.

\begin{lemma}[Continuous observations and retained information]
\label{lem:factor}
Let a countable dense additive subgroup $G\subset\R^r$ act by
measure-preserving transformations on a compact metric space with
a completed Borel probability $\mu$. For its translation operators
$T_g$, put
\[
 H_G=\{v\in L^2(\mu):\norm{T_gv-v}_2\to0
                 \text{ as }g\in G,\ g\to0\text{ in }\R^r\}.
\]
There is a completed sub-sigma-algebra $\mathcal F_G$ of measurable
events such that $H_G=L^2(\mathcal F_G)$. The orthogonal projection
onto $H_G$ is conditional expectation
$Pv=\mathbb E[v\mid\mathcal F_G]$. In particular, it preserves
nonnegative functions and the constant $1$.
\end{lemma}

The collection $\mathcal F_G$ describes the information retained
by these observations. Conditional expectation averages a random
variable while retaining that information: for every retained
event $B$, the averages of $v$ and $Pv$ on $B$ agree. Orthogonal
projection gives the decomposition
\[
 v=Pv+(v-Pv),\qquad Pv\in H_G,\qquad v-Pv\perp H_G.
\]
Appendix~\ref{app:conditional-expectation} gives the standard
projection and conditional-expectation statements. The particular
identification in Lemma~\ref{lem:factor}, together with its frequency
description below, is proved once in
Appendix~\ref{app:continuous-proof}; compare~\cite[Lemma~4.1]{OpenAI}.

For the horizontal subspace write $P_h$ for this projection, and set
\[
 p_i=P_hf_i,\qquad b_i=f_i-p_i.
\]
Positivity and preservation of constants give
\[
 0\le p_i\le1,\qquad \sum_{i=1}^k p_i=P_h1=1
 \quad\text{almost surely}.
\]
Thus the $p_i$ are color weights, still functions of an entire
coloring. The discarded parts $b_i$ lie in $H_h^\perp$ and need
not be nonnegative.

Preserving nonnegative weights does not by itself preserve the
zeros in~\eqref{eq:original-correlation}. On a space with two equally
likely outcomes, the two complementary indicators both become
the constant $1/2$ when all information is forgotten. Their product
changes from zero to $1/4$. To justify our particular projection,
we must control the translated correlations of $b_i$.

\subsection{Characters and the rigidity assumption}

A character is a multiplicative phase attached to each displacement.
For the horizontal group it is a function $d:E\to\T$, where
$\T=\{w\in\C:|w|=1\}$, satisfying
\[
 d(z+w)=d(z)d(w).
\]
The familiar examples are the ordinary frequencies
$j(\xi)(z)=e^{i\xi\cdot z}$, with $\xi\in\R^2$ and $z$ viewed
as a real planar vector. Initially we allow \emph{all} characters,
without a Euclidean continuity requirement. Write $D$ for their
group under pointwise multiplication, with the topology of
pointwise convergence, and put
\[
 C=\{j(\xi):\xi\in\R^2\}\subset D.
\]
These are exactly the characters continuous for the Euclidean
topology on $E$. A point of $D$ is a function on spatial
displacements, not a point in the plane.

The homomorphism identities define a closed subset of the compact
countable product $\T^E$, so $D$ is compact and metrizable.
The set $C$ is Borel: its frequency parametrization is continuous,
and the images of closed frequency balls are compact and cover
$C$. Appendix~\ref{app:dual-background} supplies the details and
the classification of Euclidean-continuous characters.

The spatial rotation group $K$ acts on these functions by
\[
 (ud)(z)=d(uz)\qquad(u\in K).
\]
Thus $K$ rotates the argument, whereas $\T$ is the circle of phase
values. A law on $D$ is rotation-invariant if replacing $d$ by
$ud$ does not change its distribution for any $u\in K$.

There is a uniform probability appropriate to the compact group
$D$: its normalized Haar measure $m_D$. It is characterized by
invariance under multiplying $d$ by any fixed character $a$.
For a finite group this is the uniform distribution on its elements;
for the ordinary circle it is uniform angular measure.

\begin{theorem}[Planar character-law rigidity; external assumption]
\label{thm:rigidity}
Every rotation-invariant Borel probability measure $\nu$ on $D$
satisfying $\nu(C)=0$ equals $m_D$.
\end{theorem}

This is the statement of~\cite[Theorem~2.3]{OpenAI} that we assume
throughout the three-dimensional transfer. We do not prove it here.
Its hypothesis says that a random character is Euclidean-continuous
with probability zero; the conclusion determines its entire law.

We will use the following direct consequence of Haar invariance:
\begin{equation}\label{eq:haar}
 \int_Dd(z)\,dm_D(d)=0\qquad(z\in E\setminus\{0\}).
\end{equation}
To prove it, fix $z\ne0$ and choose the ordinary character
$a=j(\pi z/|z|^2)$, which has $a(z)=-1$. Multiplication by $a$
preserves $m_D$, so the displayed integral equals its own negative.
At $z=0$ every character has value $1$, and the integral is $1$.

\subsection{A measure recording translated correlations}

For commuting unitary matrices, one can choose common eigenvectors.
On a common eigenvector, translation by $z$ multiplies by a phase
$d(z)$; the group law forces these phases to form a character.
The correlation of a vector with its translate is then a sum of
these phases weighted by the squared lengths of its components.
The spectral theorem replaces that sum by an integral in our
Hilbert space. It does not require an eigenvector basis in
infinite dimension.

For the horizontal translations and every $v\in L^2(\mu)$ there
is a unique finite positive Borel measure $\lambda_v^h$ on $D$
such that
\begin{equation}\label{eq:spectral-interface}
 \ip{v}{T_{(z,0)}v}=\int_Dd(z)\,d\lambda_v^h(d),
 \qquad \lambda_v^h(D)=\norm{v}_2^2.
\end{equation}
More precisely, there are orthogonal projections $\Pi(S)$ for
Borel sets $S\subset D$ that extract the corresponding frequency
part of a vector, and
$\lambda_v^h(S)=\norm{\Pi(S)v}_2^2$. This is the projection-valued
form of the spectral theorem; its precise statement, for any
countable abelian translation group, is given in
Appendix~\ref{app:frequencies}; see~\cite[Theorem~D.3.1]{BHV}.
The measure $\lambda_v^h$ lives on characters and records the
correlations of one observation. The probability $\mu$ lives on
colorings. In particular, $\lambda_v^h$ need not have mass one.

We can now state the frequency description of the \emph{same}
continuous-observation subspace used in Lemma~\ref{lem:factor}.

\begin{lemma}[Frequency description of horizontal continuity]
\label{lem:frequencies}
For the horizontal action, $H_h$ is the range of $\Pi(C)$, and
$P_h=\Pi(C)$. Equivalently, $v\in H_h$ exactly when
$\lambda_v^h(D\setminus C)=0$. Consequently
\[
 b\perp H_h\quad\Longrightarrow\quad
 \lambda_b^h(C)=\norm{P_hb}_2^2=0.
\]
\end{lemma}

Its proof is the frequency part of the combined proof in
Appendix~\ref{app:continuous-proof}. The event description supplies
positivity of $P_h$; the frequency description now permits us to
use rigidity on what $P_h$ discards.

\subsection{Keeping the vertical frequency in the argument}

Fix a color $i$. Recall that $b_i=f_i-p_i$ is a function of the
random coloring. We need its correlation with translations that
have both horizontal and vertical components. Let $D_v$ be the
compact group of all characters $e:F\to\T$, with no continuity
assumption. Every character of $A=E\times F$ has the form
$(z,t)\mapsto d(z)e(t)$. The spectral theorem for this joint
translation action gives a finite positive measure $\lambda_i$
on $D\times D_v$ satisfying
\begin{equation}\label{eq:joint-correlation}
 \ip{b_i}{T_{(z,t)}b_i}
 =\int_{D\times D_v}d(z)e(t)\,d\lambda_i(d,e),
 \qquad \lambda_i(D\times D_v)=\norm{b_i}_2^2.
\end{equation}
Taking $t=0$ shows that its horizontal marginal
$S\mapsto\lambda_i(S\times D_v)$ has the coefficients of
$\lambda_{b_i}^h$. Uniqueness and Lemma~\ref{lem:frequencies}
therefore give
\begin{equation}\label{eq:no-continuous-horizontal-mass}
 \lambda_i(C\times D_v)=0.
\end{equation}

We also need rotational invariance of this joint measure. Every
translation preserves $H_h$, because it commutes with the
horizontal translations. Rotations preserve $H_h$ because they
send small horizontal displacements to small horizontal
displacements. The inverse operators have the same property.
They consequently preserve $H_h^\perp$ as well and commute with
$P_h$; Appendix~\ref{app:commutation} states this elementary
projection fact. A rotation fixes the observation $f_i$ at the
origin, so it fixes both $p_i$ and $b_i$.

In operator notation, set $U_uv(\omega)=v(R_{u^{-1}}\omega)$,
where $R_{u^{-1}}$ acts on colorings as above. Then
\[
 U_uT_{(z,t)}U_u^{-1}=T_{(uz,t)},\qquad U_ub_i=b_i.
\]
Unitarity implies that the correlation in~\eqref{eq:joint-correlation}
is unchanged when $(z,t)$ is replaced by $(uz,t)$. Uniqueness of
the spectral measure says that $\lambda_i$ is unchanged under
$(d,e)\mapsto(ud,e)$: its distribution after this map has the same
Fourier coefficients. The vertical character is left unchanged.

Now fix \emph{any} Borel set $B\subset D_v$ and define a measure
on the horizontal characters by
\[
 \nu_B(S)=\lambda_i(S\times B)\qquad(S\subset D\text{ Borel}).
\]
It is finite, positive, rotation-invariant, and gives $C$ mass
zero by~\eqref{eq:no-continuous-horizontal-mass}. Let
$c_B=\nu_B(D)=\lambda_i(D\times B)$. If $c_B>0$, apply the
rigidity assumption to the probability $\nu_B/c_B$. If $c_B=0$,
positivity already gives $\nu_B=0$. Both cases yield
\begin{equation}\label{eq:rectangle-factorization}
 \lambda_i(S\times B)=m_D(S)\lambda_i(D\times B)
 \qquad(S\subset D,\ B\subset D_v\text{ Borel}).
\end{equation}
Put $\eta_i(B)=\lambda_i(D\times B)$. Finite measures agreeing
on all Borel rectangles agree on the sigma-algebra generated by
those rectangles. Here the factors are compact metric spaces,
so this is the Borel sigma-algebra of their product. Hence
\[
 \lambda_i=m_D\otimes\eta_i.
\]
Applying rigidity only to the horizontal marginal would not
justify this conclusion: two copies of the same uniform random
variable have uniform marginals but are dependent. The argument
for every vertical set $B$ is what proves the product identity.

Substituting it into~\eqref{eq:joint-correlation} and using
Haar cancellation gives
\[
 \ip{b_i}{T_{(z,t)}b_i}
 =\left(\int_Dd(z)\,dm_D(d)\right)
  \left(\int_{D_v}e(t)\,d\eta_i(e)\right)=0
 \qquad(z\ne0,\ t\in F).
\]
The measure $\eta_i$ has mass $\norm{b_i}_2^2$; no continuity
or other restriction on its vertical frequencies was needed.

Finally, translations preserve both horizontal subspaces, as
observed above. Therefore both translated mixed terms vanish:
\[
 \ip{p_i}{T_gb_i}=0,\qquad \ip{b_i}{T_gp_i}=0.
\]
For a nonvertical algebraic unit vector $g=(z,t)$, expand
$f_i=p_i+b_i$ to obtain
\begin{align*}
 0=\ip{f_i}{T_gf_i}
 &=\ip{p_i}{T_gp_i}+\ip{p_i}{T_gb_i}
   +\ip{b_i}{T_gp_i}+\ip{b_i}{T_gb_i}\\
 &=\ip{p_i}{T_gp_i}.
\end{align*}
We have proved
\begin{equation}\label{eq:horizontal-correlation}
 \ip{p_i}{T_{(z,t)}p_i}=0
 \qquad\bigl((z,t)\in A,\ |z|^2+t^2=1,\ z\ne0\bigr).
\end{equation}
Thus the horizontally continuous color weights retain every
required nonvertical zero correlation. At $z=0$ the Haar integral
is $1$, so the vertical unit displacements remain to be treated.

\section{Using the geometry of the sphere}\label{sec:sphere}

The weights $p_i=P_hf_i$ are continuous under small horizontal
translations in the $L^2$ norm. Vertical continuity is still missing.
To obtain the estimate needed for vertical averaging, we use the
geometry of the unit sphere: a small change in height can be
compensated by a small change in horizontal length.

Fix an algebraic unit vector $u_0=(z_0,t_0)\in A$ with $z_0\ne0$.
For all sufficiently small $s\in F$, put
\[
 z_s=z_0\sqrt{\frac{1-(t_0+s)^2}{1-t_0^2}}.
\]
The denominator is positive because $|z_0|^2=1-t_0^2>0$; the
numerator remains positive for small $s$. The square root is algebraic,
so $z_s\in E\setminus\{0\}$. Moreover, $z_s\to z_0$ and
\[
 |z_s|^2+(t_0+s)^2
 =|z_0|^2\frac{1-(t_0+s)^2}{1-t_0^2}+(t_0+s)^2=1.
\]
Thus $(z_s,t_0+s)$ is again a nonvertical algebraic unit vector.
Its correlation is zero by \eqref{eq:horizontal-correlation}.

Compare it with $(z_0,t_0+s)$, which has the same height. Factoring
the difference of their translation operators gives
\[
 T_{(z_0,t_0+s)}-T_{(z_s,t_0+s)}
 =T_{(z_s,t_0+s)}\bigl(T_{(z_0-z_s,0)}-\Id\bigr).
\]
Subtract the zero correlation at $(z_s,t_0+s)$ and apply
Cauchy--Schwarz and preservation of the norm:
\begin{equation}\label{eq:local}
 \left|\ip{p_i}{T_{(z_0,t_0+s)}p_i}\right|
 \leq\norm{p_i}_2
       \norm{\bigl(T_{(z_0-z_s,0)}-\Id\bigr)p_i}_2
 \longrightarrow0\qquad(s\in F,\ s\to0).
\end{equation}
Only horizontal continuity has been used. In particular, for every
error tolerance there is a neighborhood of zero in which this bound
holds for \emph{every} algebraic $s$. Equivalently,
\[
 \sup_{s\in F,\ |s|\leq\delta}
 \left|\ip{p_i}{T_{(z_0,t_0+s)}p_i}\right|
 \longrightarrow0\qquad(\delta\downarrow0).
\]
This controls the full set of shifts in a small averaging kernel;
it does not assert that $p_i$ is vertically continuous.
We now average the vertical translations $T_{(0,t)}$ over small algebraic
heights $t$. The output must still be a function of the random coloring,
so this is an average of vectors in $L^2(\mu)$. At this stage the orbit is
defined only for $t\in F$, and we have not proved its continuity. The
following lemma supplies both the averaging operation and its norm limit.

\begin{lemma}[Averaging algebraic heights]\label{lem:smoothing}
Let $(V_t)_{t\in F}$ be linear operators on a complex Hilbert space $H$
which preserve inner products and satisfy $V_0=\Id$ and
$V_sV_t=V_{s+t}$. No continuity in $t$ is assumed. Let $Q$ be the
orthogonal projection onto the closed subspace
\[
 H_v=\left\{v\in H:\norm{V_tv-v}\longrightarrow0
       \text{ as }t\in F,\ t\longrightarrow0\text{ in }\R\right\}.
\]
There is a complex-linear functional $I_{\C}$ on all bounded functions
$h:F\to\C$ which vanish outside some bounded real interval. It is
positive, meaning $I_{\C}(h)\geq0$ if $h\geq0$, and satisfies
\[
 I_{\C}(h(\,\cdot+s))=I_{\C}(h)\quad(s\in F),
 \qquad
 I_{\C}(\varphi|_F)=\int_{\R}\varphi(t)\,dt
       \quad(\varphi\in C_c(\R)).
\]
Here $C_c(\R)$ denotes the continuous functions which vanish outside a
bounded interval. For every bounded map $W:F\to H$ vanishing outside
such an interval, there is a unique vector $I_H(W)$ such that
\[
 \ip{w}{I_H(W)}
   =I_{\C}\bigl(t\mapsto\ip{w}{W(t)}\bigr)
       \quad(w\in H),
 \qquad
 \norm{I_H(W)}\leq I_{\C}\bigl(t\mapsto\norm{W(t)}\bigr).
\]
Inner products are linear in their second argument. The vector map is
linear and translation-invariant in $t$, and every bounded linear
operator $L:H\to H$ can be moved through it:
\[
 L I_H(W)=I_H(t\mapsto L W(t)).
\]
For every bounded $h:F\to\C$ and every nonnegative
$\varphi\in C_c(\R)$ of integral one,
\begin{equation}\label{eq:height-support-bound}
 \left|I_{\C}(\varphi h)\right|
 \leq\sup_{t\in F\cap\supp\varphi}|h(t)|.
\end{equation}
The product $\varphi h$ belongs to the domain of $I_{\C}$ even when
$h$ itself does not.

For $\varphi\in C_c(\R)$ define
\[
 S_\varphi v=I_H(t\mapsto\varphi(t)V_tv).
\]
This operator has range in $H_v$, vanishes on $H_v^\perp$, and satisfies
$\norm{S_\varphi}\leq\int_{\R}|\varphi|$. If
$\varphi_\varepsilon\in C_c(\R)$, $\varphi_\varepsilon\geq0$,
$\int_{\R}\varphi_\varepsilon=1$, and
$\supp\varphi_\varepsilon\subset[-\delta_\varepsilon,
\delta_\varepsilon]$ with $\delta_\varepsilon\to0$, then
\begin{equation}\label{eq:smoothing}
 \norm{S_{\varphi_\varepsilon}v-Qv}\longrightarrow0
 \qquad(v\in H).
\end{equation}
Finally, if $H=L^2(\Omega,\mu)$ for a completed Borel probability on a
compact metric space and the $V_t$ come from measure-preserving
transformations, then $Q$ preserves nonnegative functions and $Q1=1$.
\end{lemma}

Appendix~\ref{app:height-averaging} proves the lemma. In the application
here, $H=L^2(\mu)$ and $V_t=T_{(0,t)}$. Thus the limit is the vertical
continuous-observation projection, while
\eqref{eq:height-support-bound} controls an average using only the
algebraic heights in its kernel's support. For example,
$\varphi_\varepsilon(t)=\varepsilon^{-1}
\max\{1-|t|/\varepsilon,0\}$ is a permitted kernel: its graph is a
triangle of area one supported on $[-\varepsilon,\varepsilon]$.

Apply Lemma~\ref{lem:smoothing} with $V_t=T_{(0,t)}$ and set
\[
 a_i=Qp_i.
\]
In this probability model $Q$ preserves nonnegative functions and the
constant one, so $0\leq a_i\leq1$ and $\sum_i a_i=1$ as
$L^2(\mu)$ identities. Every algebraic translation $T_g$ preserves
$H_v$, since it commutes with the vertical translations. Its inverse
$T_{-g}$ also preserves $H_v$, and consequently $T_g$ preserves
$H_v^\perp$ as well. The commutation observation in
Appendix~\ref{app:commutation} therefore gives
\[
 QT_g=T_gQ\qquad(g\in A).
\]
In particular, horizontal continuity survives this second projection:
\[
 \norm{T_{(z,0)}a_i-a_i}_2
 =\norm{Q(T_{(z,0)}p_i-p_i)}_2
 \leq\norm{T_{(z,0)}p_i-p_i}_2\longrightarrow0.
\]
The new weights are vertically continuous by their membership in $H_v$.

We must also check that the averaging preserves the zero correlation
at $u_0$. Define the bounded scalar function
\[
 c_i(s)=\ip{p_i}{T_{(z_0,t_0+s)}p_i}\qquad(s\in F).
\]
It is bounded by $\norm{p_i}_2^2$ everywhere, and
\eqref{eq:local} controls it near zero. Moving the bounded operator
$T_{u_0}$ through the vector integral and then testing against $p_i$
gives the exact identity
\[
 \begin{aligned}
 \ip{p_i}{T_{u_0}S_{\varphi_\varepsilon}p_i}
 &=\ip{p_i}{I_H\bigl(s\mapsto
       \varphi_\varepsilon(s)T_{u_0}T_{(0,s)}p_i\bigr)}\\
 &=I_{\C}\bigl(\varphi_\varepsilon c_i\bigr).
 \end{aligned}
\]
The right-hand side is defined even though $c_i$ need not be continuous
or supported on a bounded interval: multiplying by the kernel puts it
in the domain of $I_{\C}$. The support bound
\eqref{eq:height-support-bound} and \eqref{eq:local} imply
\[
 \left|I_{\C}\bigl(\varphi_\varepsilon c_i\bigr)\right|
 \leq\sup_{s\in F,\ |s|\leq\delta_\varepsilon}|c_i(s)|
 \longrightarrow0.
\]
On the other hand, \eqref{eq:smoothing} says that
$S_{\varphi_\varepsilon}p_i\to Qp_i=a_i$ in $L^2$ norm. Passing to
the limit in the scalar product yields
\[
 \ip{p_i}{T_{u_0}a_i}=0.
\]
Since $T_{u_0}a_i\in H_v$ and $p_i-a_i\in H_v^\perp$, the
orthogonal part contributes nothing. Hence
\[
 \ip{a_i}{T_{u_0}a_i}=0
 \qquad(u_0\in A,\ |u_0|=1,\ u_0\text{ nonvertical}).
\]

Combining the horizontal and vertical estimates gives joint continuity:
\[
 \norm{T_{(z,t)}a_i-a_i}_2
 \leq\norm{T_{(z,0)}a_i-a_i}_2
     +\norm{T_{(0,t)}a_i-a_i}_2
 \longrightarrow0
 \qquad((z,t)\in A,\ (z,t)\to0).
\]
Translations extend from the dense subgroup $A$ to all real parameters
on the subspace of jointly continuous vectors. Concretely, for
$g_n\in A$ tending to $x\in\R^3$, define $T_xv$ as the $L^2$ limit
of $T_{g_n}v$. Joint continuity and
$\norm{T_{g_n}v-T_{g_m}v}_2=\norm{T_{g_n-g_m}v-v}_2$ make the
sequence Cauchy and the limit independent of the approximation.
The extensions preserve inner products and the translation group law,
and their orbits are norm-continuous. All real translates $T_xa_i$
still lie in $[0,1]$ and sum to one, because these properties are
preserved by $L^2$ limits. This extends operators on observations;
it requires no pointwise action of real translations on $\Omega$.

Nonvertical algebraic unit vectors are dense in the whole sphere.
Indeed, the rational stereographic parametrization
\[
 (r,s)\longmapsto
 \frac{(2r,2s,1-r^2-s^2)}{1+r^2+s^2},
 \qquad (r,s)\in\Q^2\setminus\{(0,0)\},
\]
uses rational, hence algebraic, coordinates and has nonzero horizontal
part. Rational approximation gives every nonpolar direction in its
closure. Taking $(r,s)\to(0,0)$ or $r^2+s^2\to\infty$ gives the two
poles. Continuity now extends the zero correlation to every real unit
direction, and preservation of inner products gives
\begin{equation}\label{eq:continuous-correlation}
 \ip{T_xa_i}{T_{x+u}a_i}=\ip{a_i}{T_ua_i}=0
 \qquad(x\in\R^3,\ u\in S^2).
\end{equation}

\section{Choosing a measurable coloring}\label{sec:sample}

Each $a_i$ is still a function of the random coloring, and
$x\mapsto T_xa_i$ is a continuous family of $L^2(\mu)$ classes.
To select a realization $\omega$, we need actual values depending
measurably on both $x$ and $\omega$. Independent choices of
representatives at uncountably many positions need not have that
property. The following lemma supplies these values.

\begin{lemma}[Jointly Borel representatives]\label{lem:representatives}
Let $\Omega$ be a compact metric space with a Borel probability
measure $\mu$, and use the completion of $\mu$ when forming
$L^2(\mu)$. Suppose $x\mapsto u_x$ is a norm-continuous map from
$\R^3$ to $L^2(\mu)$ and $0\leq u_x\leq1$ almost surely for every
$x$. Then there are a jointly Borel function
$H:\R^3\times\Omega\to[0,1]$ and a Borel Lebesgue-null set
$N\subset\R^3$ such that $H(x,\cdot)$ represents $u_x$ in
$L^2(\mu)$ for every $x\notin N$. In particular, for every fixed
$\omega\in\Omega$, the spatial function $x\mapsto H(x,\omega)$ is
Borel measurable.
\end{lemma}

Appendix~\ref{app:representatives} gives the full proof. Its construction
approximates the continuous map on bounded spatial regions by finite
Borel step functions, chooses Borel representatives of their finitely
many values, and takes a limit on the product space. The resulting
exception is a spatial null set; no assertion of simultaneous
pointwise equality at every $x$ is needed.

Apply the lemma to each $x\mapsto T_xa_i$, obtaining $H_i$.
Take the finite union of the exceptional sets and call it $N$.
For $x\notin N$, all sections represent the corresponding translated
weights, and thus
\[
 \sum_{i=1}^k H_i(x,\omega)=1
 \quad\text{for }\mu\text{-almost every }\omega.
\]
For a fixed unit direction $u$, the starting points at which either
section $H_i(x,\cdot)$ or $H_i(x+u,\cdot)$ might fail to represent
its vector lie in $N\cup(N-u)$. This set is null by translation
invariance of Lebesgue measure. Integrating in $u$ shows that the
exceptional endpoint pairs have zero $dx\,d\sigma(u)$ measure.
At every other pair, \eqref{eq:continuous-correlation} says
\[
 \int_\Omega H_i(x,\omega)H_i(x+u,\omega)\,d\mu(\omega)=0.
\]

The functions in this integral are nonnegative and jointly Borel.
Tonelli's theorem therefore permits the following actual exchange of
the three integrations, for each finite $R>0$:
\begin{equation}\label{eq:sample-triple-integral}
 \begin{aligned}
 &\int_\Omega\left[
   \int_{B(0,R)}\int_{S^2}
   H_i(x,\omega)H_i(x+u,\omega)\,d\sigma(u)\,dx
   \right]d\mu(\omega)\\
 &\qquad=\int_{B(0,R)}\int_{S^2}\left[
   \int_\Omega H_i(x,\omega)H_i(x+u,\omega)\,d\mu(\omega)
   \right]d\sigma(u)\,dx=0.
 \end{aligned}
\end{equation}
Thus the nonnegative random variable in the first brackets is zero
for $\mu$-almost every $\omega$. The sum-one identity gives, by the
same theorem,
\[
 \int_\Omega\left[
  \int_{B(0,R)}\left|\sum_{i=1}^kH_i(x,\omega)-1\right|\,dx
 \right]d\mu(\omega)=0.
\]
Its bracketed random variable also vanishes almost surely.

Choose one $\omega$ for which these conclusions hold for all
$i\in\{1,\ldots,k\}$ and all positive integer radii $R$.
This involves only countably many events of probability one, so their
intersection still has probability one. Put $h_i(x)=H_i(x,\omega)$.
These are Borel functions of the spatial variable. The integer balls
cover space, so their sum is one outside the Borel null set
\[
 Z=\left\{x:\sum_{i=1}^kh_i(x)\ne1\right\}.
\]
Moreover, their same-color product integrals vanish on every bounded
ball, because each such ball is contained in one of the integer balls.

At $x\notin Z$, choose the least index $i$ with $h_i(x)\geq1/k$;
one exists since the $k$ nonnegative weights sum to one. On $Z$ choose
color one. The resulting color classes $B_1,\ldots,B_k$ form a Borel
partition. If $x,x+u\notin Z$ both receive color $i$, then
\[
 \one_{B_i}(x)\one_{B_i}(x+u)
 \leq k^2h_i(x)h_i(x+u).
\]
The right side has zero integral over $B(0,R)\times S^2$. The pairs
with an endpoint in $Z$ also have zero measure: for each fixed $u$
their starting points lie in the null set $Z\cup(Z-u)$, and Tonelli
integrates this fact over the directions. Thus the partition satisfies
\eqref{eq:weak} for every color and every radius. This proves
Theorem~\ref{thm:transfer}. \hfill$\square$

\section{Removing the exceptional pairs}\label{sec:completion}

We now return to ordinary subsets of Euclidean space and integration
against spatial Lebesgue measure. The preceding argument produced a
measurable partition with almost no monochromatic unit pairs, measured
in position and unit direction. We will turn it into a coloring that
avoids every such pair. This step is independent of the rigidity
assumption and works in every positive dimension.

The overlap and density-point arguments below follow the methods
in~\cite[proof of Lemma~B.2]{Davies}
and~\cite[proof of Lemma~5.6]{DOV}.
The finite-translation completion is the argument
of~\cite[Lemma~B.1]{Davies}, whose stated scope is $d\ge2$.
We include the proof for all $d\ge1$. In dimension one the uniform
measure on the unit sphere $\{-1,1\}$ assigns mass $1/2$ to each direction.

\begin{lemma}[Completion of a weak coloring]\label{lem:completion}
For every integer $d\ge1$ and every positive integer $k$, a weakly
proper measurable $k$-coloring of $\R^d$ can be replaced by a proper
Lebesgue-measurable $k$-coloring.
\end{lemma}

\begin{proof}
Let $B_1,\ldots,B_k$ be the given color classes. Choose Borel
representatives modulo null sets, make them disjoint in order, and
assign any remaining points to the last class. This gives a Borel
partition differing from the original only on a null set. We continue
to call its classes $B_i$. As observed after~\eqref{eq:weak}, this
change preserves weak properness: for each fixed direction $u$,
the pairs whose initial or final point meets a spatial null set have
only a null set of initial points.

\paragraph{An overlap identity in every direction.}
Define
\[
 g_i(x)=e^{-\norm{x}^2}\one_{B_i}(x),
 \qquad
 G_i(v)=\int_{\R^d}g_i(x)g_i(x+v)\,dx.
\]
The positive Gaussian puts $g_i$ in $L^2(\R^d)$ even when $B_i$ has
infinite volume. Its positivity also means that $G_i(v)>0$ exactly
when $B_i$ and $B_i-v$ overlap on a set of positive volume.

Ordinary spatial translations are continuous in $L^2$: if
$g\in L^2(\R^d)$, then $\norm{g(\cdot+h)-g}_2\to0$ as $h\to0$.
To see this, approximate $g$ in $L^2$ by a continuous function of
compact support. Uniform continuity proves the assertion for the
approximation, and translations preserve the norm of the error.
Consequently, Cauchy--Schwarz gives
\[
 |G_i(v)-G_i(w)|
 \le \norm{g_i}_2
       \norm{g_i(\cdot+v)-g_i(\cdot+w)}_2,
\]
so the nonnegative function $G_i$ is continuous.

Weak properness on all positive integer-radius balls, followed by
Tonelli's theorem, gives
$\one_{B_i}(x)\one_{B_i}(x+u)=0$ for almost every position--direction
pair. Multiplying by the Gaussian factors and integrating shows that
$G_i(u)=0$ for almost every unit direction. In dimensions $d\ge2$,
every nonempty relatively open part of the unit sphere has positive
surface measure. A continuous nonnegative function on that sphere
which vanishes almost everywhere must therefore vanish everywhere.
For $d=1$, almost-everywhere vanishing already includes both
directions, since each has positive mass. Thus, in all dimensions,
\begin{equation}\label{eq:completion-overlap}
 G_i(u)=0\qquad(\norm{u}=1).
\end{equation}

\paragraph{Independent sets of density-one points.}
A point $x$ has \emph{density one} in a measurable set $B$ if
\[
 \lim_{r\downarrow0}
 \frac{\vol(B\cap B(x,r))}{\vol(B(x,r))}=1.
\]
This says that $B$ occupies nearly all sufficiently small balls
about $x$; it requires neither that $x\in B$ nor that $B$ contain
a neighborhood of $x$. Let $C_i$ be the set of density-one points
of $B_i$.

These sets are Borel. Indeed, for each fixed $r>0$ the ball average
\[
 A_{i,r}(x)=\frac{\vol(B_i\cap B(x,r))}{\vol(B(0,r))}
\]
is continuous in $x$: its change between two centers is bounded
by the volume of the symmetric difference of their balls, divided
by $\vol(B(0,r))$, and that bound tends to zero as the centers
approach one another. Moreover,
\[
 C_i=\{x:\lim_{n\to\infty}A_{i,1/n}(x)=1\}.
\]
The sequence $1/n$ detects the full density limit because any
intermediate radius lies between $1/(n+1)$ and $1/n$, whose ball
volumes have ratio tending to one. The displayed limit condition
is Borel, since it uses only countably many continuous functions.

The Lebesgue density theorem states that almost every point of
$B_i$ has density one in $B_i$, and almost every point of its
complement has density zero in $B_i$. It follows that
\[
 \vol(C_i\mathbin{\triangle}B_i)=0.
\]
The $C_i$ are disjoint: two disjoint classes cannot both occupy a
proportion tending to one of the same shrinking ball. Their union
$C_*=\bigcup_{i=1}^k C_i$ is conull, meaning that its complement has
measure zero. Also, every point of $C_i$ has density one in $C_i$
itself, because changing a set on a null set leaves every ball
volume unchanged.

Each $C_i$ is independent in the unit-distance graph. Suppose,
to the contrary, that $x,x+u\in C_i$ and $\norm{u}=1$. Then both
$B_i$ and $B_i-u$ have density one at $x$. In a sufficiently small
ball about $x$, each misses less than one quarter of the volume.
Their intersection therefore occupies more than half the ball.
On this positive-volume intersection,
$g_i(q)g_i(q+u)>0$, contradicting~\eqref{eq:completion-overlap}.
We have obtained a proper coloring on $C_*$, with color $i$ on $C_i$.

\paragraph{A compatible coloring of each finite set.}
The remaining null set may have unit edges both to $C_*$ and
within itself, so we must extend the coloring while retaining its
prescribed values on $C_*$. Fix a finite set $V\subset\R^d$.
We seek one small translation $t$ such that
\begin{equation}\label{eq:completion-shift}
 x+t\in C_i\quad(x\in V\cap C_i),
 \qquad
 x+t\in C_*\quad(x\in V\setminus C_*).
\end{equation}
We could then color each $x\in V$ by the color of $x+t$ in $C_*$.
This would respect every prescribed color on $V\cap C_*$, and
it would be proper because translation preserves distances.

Let $m=|V\cap C_*|$. If $m>0$, density one at each prescribed
point gives a radius $r>0$ such that, for every $x\in V\cap C_i$,
the bad shifts
$\{t\in B(0,r):x+t\notin C_i\}$ occupy less than
$1/(2m)$ of the ball. A common radius exists because there are
only finitely many prescribed points. The union bound leaves
more than half the ball satisfying all these requirements.
If $m=0$, begin with any ball about zero.
For each remaining $x\in V\setminus C_*$, the shifts failing
$x+t\in C_*$ form a null set. Removing these finitely many null
sets still leaves a suitable $t$. The radius can be chosen
arbitrarily small, but it and the translation may depend on $V$.

\paragraph{Compactness with the prescribed colors.}
Consider the compact product
$X=\{1,\ldots,k\}^{\R^d}$, with the finite color set discrete.
An element of $X$ assigns a color to every spatial point.
For every unit edge $\{x,y\}$ impose the constraint $c(x)\ne c(y)$;
for every $x\in C_i$ impose $c(x)=i$.
Each constraint defines a closed subset of $X$, because it
depends on finitely many coordinates. Every finite collection
of constraints involves a finite vertex set $V$, and the
translation argument gives a coloring of $V$ satisfying them.
Extend that finite assignment arbitrarily to the other
coordinates. Thus every finite collection of these closed sets
has nonempty intersection. Compactness gives a point in their
total intersection: a proper coloring of all $\R^d$ retaining
the color of every point in $C_*$.
This is the graph-compactness argument of
de Bruijn and Erd\H{o}s~\cite{dBE}, with fixed-color constraints
included. No limiting translation is needed.

The final color classes have the form
\[
 \widetilde B_i=C_i\cup N_i,
 \qquad N_i\subseteq\R^d\setminus C_*.
\]
The sets $C_i$ are Borel, and the remainder is a Borel null set.
Every subset $N_i$ of that remainder is measurable for completed
Lebesgue measure. Hence the $\widetilde B_i$ are Lebesgue measurable,
as required. The compactness argument does not assert that these
final classes are Borel.
\end{proof}

\section{Consequences}\label{sec:consequences}

\begin{proof}[Proof of Theorem~\ref{thm:main}]
Assume the character-law rigidity statement in
Theorem~\ref{thm:rigidity}. Given a proper $k$-coloring of $\R^3$,
Theorem~\ref{thm:transfer} produces a weakly proper measurable
$k$-coloring, and Lemma~\ref{lem:completion} completes it to a
proper Lebesgue-measurable $k$-coloring. Both steps retain the
number of colors. Thus $\chi_m(\R^3)\le\chi(\R^3)$; the reverse
inequality follows from the definitions.
\end{proof}

\paragraph{Upper density and the numerical bounds.}
For a measurable independent set $B\subset\R^3$, define
\[
 \overline\delta(B)=\limsup_{L\to\infty}
   \frac{\vol(B\cap[-L,L]^3)}{(2L)^3},
 \qquad
 m_1(\R^3)=
   \sup_{\substack{B\text{ measurable}\\B\text{ independent}}}
   \overline\delta(B).
\]
These are the centered-cube upper density and its supremum over
measurable sets avoiding unit distance.
The independent sets $C_1,\ldots,C_k$ obtained in the completion
proof already form a partition up to a null set. Their volume
fractions in each centered cube therefore sum to one. Taking
upper limits and using that the sum is finite gives
\begin{equation}\label{eq:density-bound}
 1\le
 \sum_{i=1}^k\limsup_{L\to\infty}
       \frac{\vol(C_i\cap[-L,L]^3)}{(2L)^3}
 \le k\,m_1(\R^3).
\end{equation}
No density limit for an individual class is needed.

De Oliveira Filho and Vallentin report the upper bound
$m_1(\R^3)\le0.165609$ in~\cite[Table~3.1]{OFV},
with the same centered-cube convention. Since $0.165609<1/6$,
\eqref{eq:density-bound} rules out a weakly proper measurable
coloring with at most six colors. In particular,
$\chi_m(\R^3)\ge7$. 

\begin{corollary}\label{cor:bounds1}
Under the character-law rigidity assumption of Theorem~\ref{thm:rigidity},
\[
 7\leq\chi(\R^3)=\chi_m(\R^3)\leq15.
\]
\end{corollary}

\begin{proof}
Under the character-law rigidity assumption,
Theorem~\ref{thm:main} transfers the measurable lower bound $7$
to the ordinary chromatic number. The upper bound $15$ comes
from the geometric constructions of Coulson and of
Radoi\v{c}i\'{c} and T\'oth~\cite{Coulson,RT}.
For the measurable upper bound, the latter construction uses
a countable family of Borel tiles; at points belonging to
several tiles, choosing the least available tile color gives
Borel color classes. Thus, under the stated rigidity assumption,
\[
 7\le\chi(\R^3)=\chi_m(\R^3)\le15.
\]
\end{proof}

\begin{corollary}[A separate planar consequence]\label{cor:plane}
Assume the planar weak-transfer statement
of~\cite[Theorem~1.3]{OpenAI}: for every positive integer $k$,
a proper $k$-coloring of $\R^2$ gives a weakly proper measurable
$k$-coloring of $\R^2$. Then
\[
 \chi(\R^2)=\chi_m(\R^2).
\]
\end{corollary}

\begin{proof}
Apply the stated planar weak transfer, followed by the
unconditional Lemma~\ref{lem:completion} with $d=2$.
Neither step increases the number of colors, and the reverse
comparison holds by definition.
\end{proof}

The planar statement uses its own weak-transfer assumption,
separately from the character-law rigidity input to the
three-dimensional argument. The latter gives the displayed
range, not an exact chromatic number or an improved upper bound.
Under that rigidity assumption, graph compactness also guarantees
a finite unit-distance graph in $\R^3$ which is not six-colorable;
the proof supplies no explicit finite seven-chromatic example.

\appendix
\numberwithin{equation}{section}

\section{Probability and functions of a random coloring}\label{app:probability}

The probability measure $\mu$ in the proof describes a random entire
coloring $\omega\in\Omega$. It differs from spatial volume $dx$ on
$\R^3$ and from the normalized area measure $\sigma$ on the sphere of
directions. This appendix records the background for averaging
colorings, treating their observations as vectors, and finally giving
those vectors jointly measurable values.

\subsection{Events, Borel sets, and null sets}

A \emph{$\sigma$-algebra} on a set is a collection of subsets closed
under complements and countable unions and containing the whole set.
Its members are the measurable sets, or \emph{events} in a probability
space. A probability measure assigns them numbers in $[0,1]$, has
total mass $1$, and is countably additive: the probability of a disjoint
countable union is the sum of the probabilities. A function is
measurable if inverse images of Borel sets are measurable.

On a topological space, the \emph{Borel sets} form the smallest
$\sigma$-algebra containing the open sets. Completing a Borel measure
adds every subset of every Borel null set, together with the sets
obtained by such null modifications. A function measurable for the
completion agrees almost everywhere with a Borel function; the
representative construction below explains the instance we use.
Lebesgue measure is completed Borel volume. Thus arbitrary assignments
on a spatial null set preserve Lebesgue measurability, as needed in
the final coloring, although they need not preserve Borel measurability.

\emph{Almost everywhere} means outside a measurable set of measure
zero; \emph{conull} means that the complement has measure zero.
For $\mu$ one also says \emph{almost surely}. Countably many
almost-sure requirements hold simultaneously almost surely. The
corresponding assertion fails for an uncountable family: each
$[0,1]\setminus\{t\}$ has Lebesgue measure $1$, but their intersection
over all $t\in[0,1]$ is empty. The sampling argument therefore makes
countably many integrated requirements before choosing one coloring.

\subsection{Existence of finite averaging sets}\label{app:finite-averaging}

For a discrete group $G$, write $\ell^\infty(G)$ for its bounded
complex-valued functions. An \emph{invariant mean} is a linear
functional $m:\ell^\infty(G)\to\C$ satisfying
\[
 m(1)=1,\qquad f\geq0\Longrightarrow m(f)\geq0,\qquad
 m(f(g\,\cdot))=m(f)\quad(g\in G).
\]
It is a normalized positive average unchanged by translation, and
$|m(f)|\leq\norm{f}_\infty$. A group admitting such a mean is called
\emph{amenable}. A mean need not come from a countably additive
probability measure. For example, on an infinite countable group a
translation-invariant probability on all subsets would have to give
every singleton the same mass, which is incompatible with countable
additivity and total mass $1$.

Two standard facts supply amenability here: every abelian group is
amenable, and a group with an amenable normal subgroup and amenable
quotient is amenable. These are Theorem~G.2.1 and
Proposition~G.2.2(ii) of~\cite{BHV}, respectively. The cited proposition
requires the normal subgroup to be closed; that condition is automatic
for the discrete groups used here.

Recall that $F$ is the field of real algebraic numbers,
$E=F+iF$, $A=E\times F$, and $K=\{u\in E:|u|=1\}$. The motion group
$\Gamma=A\rtimes K$ consists of maps $x\mapsto a+R_ux$, where
$R_u(z,t)=(uz,t)$; composition gives
\[
 (a,u)(b,v)=(a+R_ub,uv).
\]
The translations $(a,1)$ form a normal subgroup: conjugating a
translation by a motion is again a translation. This subgroup is
isomorphic to the abelian group $A$. The homomorphism $(a,u)\mapsto u$
has precisely the translations as its kernel, so the quotient, which
retains only the rotation parameter, is the abelian group $K$.
The two cited facts therefore make $\Gamma$ amenable. They also give
an invariant mean on the abelian group $F/\Z$, used later for height
averaging.

The usable finite-set consequence is the following specialization of
the F\o lner criterion~\cite[Theorem~G.5.1]{BHV}.
\begin{standard}[F\o lner criterion for a discrete group]
A discrete group $G$ is amenable if and only if, for every finite
$Q\subset G$ and every $\varepsilon>0$, there is a nonempty finite
$\Lambda\subset G$ such that
\[
 \frac{|g\Lambda\mathbin{\triangle}\Lambda|}{|\Lambda|}
 \leq\varepsilon\qquad(g\in Q).
\]
\end{standard}
The symbol $\triangle$ denotes symmetric difference, so the numerator
counts elements gained or lost under the indicated left translation.
The source states the criterion for locally compact groups, compact
test sets and left invariant measure. A discrete group is locally
compact, its compact sets are finite, and counting measure is left
invariant, giving exactly the statement above.

The group $\Gamma$ is countable because $F$ is countable. Enumerate it
as $g_1,g_2,\ldots$ and apply the criterion with
$Q_n=\{g_1,\ldots,g_n\}$ and $\varepsilon=1/n$. The resulting nonempty
finite sets $\Lambda_n$ satisfy
\[
 \frac{|g\Lambda_n\mathbin{\triangle}\Lambda_n|}{|\Lambda_n|}
 \longrightarrow0\qquad\text{for every fixed }g\in\Gamma.
\]
Thus each fixed motion changes a vanishing fraction of the finite
average. Countability suffices for this construction; no finite
generation or convergence uniform over all motions is asserted.

The second existence fact turns a sequence of finite coloring
distributions into an ordinary probability.
\begin{theorem}[Weak compactness of probabilities]\label{app:probability-compactness}
If $X$ is a compact metric space and $\mu_n$ are Borel probability
measures on $X$, some subsequence converges to a Borel probability
measure $\mu$ in the sense that
\[
 \int_X f\,d\mu_{n_j}\longrightarrow\int_X f\,d\mu
 \qquad\text{for every continuous }f:X\to\C.
\]
This is called \emph{weak convergence} of probability measures.
\end{theorem}

Here is the standard compactness argument. The continuous functions
on a compact metric space have a countable family dense for the
uniform norm. A diagonal subsequence makes the integrals converge
on that family; the bound $|\int f\,d\mu_n|\leq\norm{f}_\infty$
extends the limit to every continuous $f$. The resulting linear
functional $L$ is positive and satisfies $L(1)=1$. The
Riesz--Markov representation theorem represents it by a unique
countably additive Borel probability measure. This also explains why
continuous test functions determine that measure. The main proof
applies this theorem to the compact metric coloring space and uses
its finite-average estimate to establish invariance of the limit.

\subsection{Mean-square geometry and conditional expectation}\label{app:conditional-expectation}

An observation $v$ of the random coloring is a function
$v:\Omega\to\C$. Its mean-square size is
$\norm{v}_2^2=\int_\Omega|v|^2\,d\mu$. The space $L^2(\mu)$ consists
of observations with finite mean-square size, identifying functions
that agree almost surely. It is a complete inner-product space with
\[
 \ip{v}{w}=\int_\Omega\overline v\,w\,d\mu.
\]
The first argument is conjugate-linear and the second is linear.
The quantity $\norm{v-w}_2^2$ is expected squared error; for two event
indicators it equals the probability that the indicators disagree.

If $M$ is a closed linear subspace of a Hilbert space, the projection
theorem gives a unique decomposition
\[
 v=Pv+(v-Pv),\qquad Pv\in M,\quad v-Pv\perp M.
\]
The vector $Pv$ minimizes $\norm{v-w}$ over $w\in M$. The operator
$P$ is linear, fixes $M$, satisfies $P^2=P=P^*$, and obeys
$\norm{Pv}\leq\norm{v}$. The adjoint notation uses the convention
$\ip{w}{Lv}=\ip{L^*w}{v}$. A projection need not preserve
nonnegative functions; the projections in this paper have that
additional property because they are conditional expectations.

A sub-$\sigma$-algebra $\mathcal G$ describes the event information
that is retained. If the retained information reports the cell of a
finite measurable partition $E_1,\ldots,E_r$, conditional expectation
replaces $v$ on each positive-probability cell $E_j$ by its average
$\mu(E_j)^{-1}\int_{E_j}v\,d\mu$. The following theorem extends that
operation to any collection of retained events.

\begin{theorem}[Conditional expectation]\label{thm:conditional-expectation}
Let $(\Omega,\mathcal F,\mu)$ be a probability space and
$\mathcal G\subseteq\mathcal F$ a sub-$\sigma$-algebra. For every
integrable $v$ there is a $\mathcal G$-measurable integrable function
$\mathbb E[v\mid\mathcal G]$, unique almost everywhere, such that
\[
 \int_B\mathbb E[v\mid\mathcal G]\,d\mu=\int_Bv\,d\mu
 \qquad(B\in\mathcal G).
\]
The operation is linear, preserves nonnegativity and the constant
$1$, and preserves expectations. For $v\in L^2(\mu)$ it is the
orthogonal projection onto the closed subspace $L^2(\mathcal G)$.
\end{theorem}

For complex observations the definition applies to real and imaginary
parts. The integral identities make the error perpendicular first to
indicators of retained events, then to their finite linear
combinations, and finally, by approximation, to all of
$L^2(\mathcal G)$. This is the link between conditional averaging and
the least mean-square error property. Identifying which event
information is retained by the continuous-translation subspace is a
separate argument, supplied in Appendix~B.

\subsection{Jointly Borel representatives and product integration}\label{app:representatives}

An $L^2$ vector records a function only up to a probability-null set.
To evaluate a continuously varying family $u_x$ at a chosen
$\omega$, we need representatives whose dependence on both variables
is measurable. ``Jointly Borel'' means Borel on
$\R^3\times\Omega$; in particular, fixing either variable gives a
Borel function of the other. We prove Lemma~\ref{lem:representatives}
with its stated almost-everywhere spatial conclusion.

\begin{proof}[Proof of Lemma~\ref{lem:representatives}]
First, every completed-measurable $[0,1]$-valued function on $\Omega$
has a Borel representative. Indeed, approximate it by finite simple
functions, for example by rounding down to successive dyadic values.
Each measurable level set differs from a Borel set by a null set.
Replacing those sets gives Borel simple functions with the same
values outside the union of countably many null exceptions. Take
their limit on the Borel set where they converge, put zero elsewhere,
and truncate to $[0,1]$. This gives the required Borel representative.

Let $Q_m=[-m,m]^3$, set $Q_0=\varnothing$, and put
$W_m=Q_m\setminus Q_{m-1}$. These are disjoint bounded Borel sets
whose union is $\R^3$. Fix $m$. Norm continuity of $x\mapsto u_x$
is uniform on the compact cube $Q_m$. For each $n$, a sufficiently
fine finite Borel partition $W_m=\bigcup_{\ell=1}^{r_n}E_{n,\ell}$,
with sample points $x_{n,\ell}\in E_{n,\ell}$, therefore gives
\[
 \norm{u_x-u_{x_{n,\ell}}}_2\leq2^{-n}
 \qquad(x\in E_{n,\ell}).
\]
Empty cells can be discarded. Choose Borel $[0,1]$-valued
representatives $h_{n,\ell}$ of these finitely many sample vectors
and define
\[
 H_n(x,\omega)=\sum_{\ell=1}^{r_n}
       \one_{E_{n,\ell}}(x)h_{n,\ell}(\omega).
\]
This is jointly Borel on $W_m\times\Omega$ and takes values in
$[0,1]$. Each section satisfies
$\norm{H_n(x,\cdot)-u_x}_2\leq2^{-n}$. Consequently
\[
 \norm{H_n-H_r}_{L^2(W_m\times\Omega)}
 \leq\vol(W_m)^{1/2}(2^{-n}+2^{-r}).
\]
Thus the sequence has a limit $V$ in the product $L^2$ space.
Choose a subsequence $H_{n_j}$ such that
$\sum_j\norm{H_{n_j}-V}_{L^2(W_m\times\Omega)}^2<\infty$.
Tonelli gives
\[
 \sum_j|H_{n_j}(x,\omega)-V(x,\omega)|^2<\infty
 \quad\text{for almost every }(x,\omega),
\]
after choosing any measurable representative of $V$. In particular,
$H_{n_j}$ converges almost everywhere on the product.

The convergence set of this sequence of Borel scalar functions is
Borel. Define $H^{(m)}$ to be its limit there and zero elsewhere.
Then $H^{(m)}$ is jointly Borel and $[0,1]$-valued. Fubini shows that,
for almost every fixed $x\in W_m$, the convergence holds for
$\mu$-almost every $\omega$. Bounded convergence on the probability
space then gives
$H_{n_j}(x,\cdot)\to H^{(m)}(x,\cdot)$ in $L^2(\mu)$.
Since the same sections tend to $u_x$ in that norm,
$H^{(m)}(x,\cdot)$ represents $u_x$ for almost every $x\in W_m$.

Finally, use $H^{(m)}$ on $W_m\times\Omega$ for each $m$. A
countable combination on these disjoint Borel regions gives one
jointly Borel function $H:\R^3\times\Omega\to[0,1]$. Enlarge each
exceptional spatial null set to a Borel null set and take their
countable union $N$. For every $x\notin N$, the section $H(x,\cdot)$
represents $u_x$, as required.
\end{proof}

The product-integration statements used here and in the sampling
argument are the following.
\begin{theorem}[Tonelli and Fubini]\label{app:tonelli-fubini}
Suppose $(X,\alpha)$ and $(Y,\beta)$ are sigma-finite measure
spaces, meaning that each is a countable union of sets of finite
measure. For nonnegative product-measurable $f$, Tonelli's theorem
gives
\[
 \int_{X\times Y}f\,d(\alpha\otimes\beta)
 =\int_X\!\left(\int_Yf(x,y)\,d\beta(y)\right)d\alpha(x)
 =\int_Y\!\left(\int_Xf(x,y)\,d\alpha(x)\right)d\beta(y),
\]
allowing infinite values. If $f$ is signed or complex and
$\int|f|<\infty$, Fubini gives the same equalities, with sections
and finite integrals interpreted almost everywhere. The statements
also apply to completed products with this almost-everywhere
interpretation.
\end{theorem}

\section{Continuous translation observations and their frequencies}\label{app:frequencies}

The continuous observations have two descriptions in the main proof. Lemma~\ref{lem:factor} describes the information they retain and shows that projection onto them preserves nonnegativity. Lemma~\ref{lem:frequencies} identifies the same subspace by its frequencies. After specifying the frequency spaces and standard theorems, we prove these two assertions together.

\subsection{The character space and the standard theorems}\label{app:dual-background}

Let $G$ be a countable dense additive subgroup of $\R^r$. Its character space is
\[
 D_G=\widehat{G_{\mathrm{disc}}}
 =\{d:G\to\T:d(g+h)=d(g)d(h)\},
 \qquad \T=\{z\in\C:|z|=1\}.
\]
The subscript $\mathrm{disc}$ records that no Euclidean continuity is imposed on these homomorphisms. A point of $D_G$ is a function on the spatial displacements in $G$. Multiplication of two such points means pointwise multiplication of their values.

The topology on $D_G$ records convergence at every fixed displacement. Here is a concrete metric and the compactness argument. Enumerate $G=\{g_1,g_2,\ldots\}$ and put
\begin{equation}\label{eq:dual-metric}
 \rho(d,e)=\sum_{n=1}^{\infty}2^{-n}|d(g_n)-e(g_n)|.
\end{equation}
Every summand before its weight is at most $2$. Consequently the tail of the series is uniformly small, and convergence at its finitely many initial coordinates implies convergence in $\rho$. Conversely, convergence in $\rho$ forces convergence at each fixed coordinate. The full space $\T^G$ is compact: from any sequence of its points, successively choose subsequences converging at $g_1,g_2,\ldots$ and take a diagonal subsequence. The equations $d(0)=1$ and $d(g+h)=d(g)d(h)$ are closed conditions, so their solution set $D_G$ is closed and hence compact. This makes $D_G$ a compact metric group.

Ordinary Euclidean frequencies give the characters
\[
 j(\xi)(g)=e^{i\xi\cdot g},\qquad \xi\in\R^r.
\]
They are exactly the characters that are continuous for Euclidean distance on $G$. Indeed, for a character $d$,
\[
 |d(g)-d(h)|=|d(g-h)-1|.
\]
Continuity at zero therefore gives uniform continuity. Approximating each real vector by elements of $G$ extends $d$ uniquely to a continuous character on $\R^r$. The usual classification of continuous characters of $\R^r$ gives $x\mapsto e^{i\xi\cdot x}$ for a unique $\xi$. Define
\[
 C_G=j(\R^r).
\]
The map $j$ is continuous because each of its coordinates is continuous, and density of $G$ makes it injective. For every integer $N\geq1$, the image of $\{\xi:|\xi|\leq N\}$ is compact and thus closed in $D_G$. Hence
\[
 C_G=\bigcup_{N\geq1}j(\{\xi:|\xi|\leq N\})
\]
is Borel. This is the measurability needed to speak of spectral mass on $C_G$. We use the parametrization and its Borel image; we do not require its inverse to identify the subspace topology with the Euclidean topology of frequencies. A measure is \emph{carried by} $C_G$ when its complement has measure zero.

\begin{theorem}[Haar probability]\label{app:haar-probability}
Every compact Hausdorff group $L$ has a unique regular Borel probability measure $m_L$ invariant under left translation. For compact groups it is also right invariant. Thus
\[
 m_L(aB)=m_L(B)=m_L(Ba),\qquad m_L(L)=1
\]
for every Borel set $B\subseteq L$ and every $a\in L$.
\end{theorem}

Here regularity means that Borel sets can be approximated in measure from within by compact sets and from outside by open sets. Haar probability is the group version of uniform probability: it gives equal masses on a finite group and uniform angle on $\T$. On $D_G$, translation means multiplying by a fixed character. Its cancellation property is proved at its use in the main text.

For a finite Borel measure $\lambda$ on $D_G$, the numbers
\[
 \widehat\lambda(g)=\int_{D_G}d(g)\,d\lambda(d),\qquad g\in G,
\]
are its Fourier coefficients. The displacement $g$ is fixed while the character $d$ varies. In the proof, these coefficients will be scalar products of observations and their translates.

\begin{theorem}[Spectral theorem for commuting translations]\label{app:spectral-theorem}
Let $g\mapsto T_g$ be a unitary representation of a countable discrete abelian group $G$ on a complex Hilbert space $H$. There is a unique projection-valued Borel measure $\Pi$ on $D_G$ such that
\[
 T_g=\int_{D_G}d(g)\,d\Pi(d).
\]
Precisely, each $\Pi(S)$ is an orthogonal projection,
\[
 \Pi(D_G)=\Id,\qquad
 \Pi(S)\Pi(B)=\Pi(S\cap B),
\]
and for disjoint Borel sets $S_n$,
\[
 \Pi\Bigl(\bigcup_n S_n\Bigr)v=\sum_n\Pi(S_n)v
 \quad(v\in H),
\]
where the sum converges in Hilbert norm. For every $v\in H$, the finite positive Borel measure
\[
 \lambda_v(S)=\ip{v}{\Pi(S)v}=\norm{\Pi(S)v}^2
\]
satisfies
\[
 \ip{v}{T_gv}=\int_{D_G}d(g)\,d\lambda_v(d),
 \qquad \lambda_v(D_G)=\norm{v}^2.
\]
It is uniquely determined by these Fourier coefficients.
\end{theorem}

This is the form of the spectral theorem in~\cite[Theorem~D.3.1]{BHV}. A projection-valued measure assigns to a frequency region $S$ the operator extracting the component in that region; its squared length is $\lambda_v(S)$. The formula allows frequencies to be distributed continuously and does not assert the existence of a basis of eigenvectors. The inner product convention throughout is conjugate-linear in the first argument and linear in the second.

The converse frequency argument uses the following theorem on ordinary Euclidean space. A function $c:\R^r\to\C$ is called \emph{positive definite} if
\[
 \sum_{j,\ell=1}^{N}\overline{\alpha_j}\alpha_\ell
       c(x_\ell-x_j)\geq0
\]
for every finite choice of points $x_j\in\R^r$ and coefficients $\alpha_j\in\C$. Equivalently, all these finite matrices are positive semidefinite.

\begin{theorem}[Bochner's theorem]\label{app:bochner}
A continuous function $c:\R^r\to\C$ is positive definite if and only if there is a finite positive Borel measure $\nu$ on $\R^r$ such that
\[
 c(x)=\int_{\R^r}e^{i\xi\cdot x}\,d\nu(\xi).
\]
The measure is unique and has total mass $c(0)$.
\end{theorem}

See~\cite[Theorem~D.2.2]{BHV}. Its role below is to turn a continuous translation correlation into a measure of ordinary Euclidean frequencies.

\subsection{One proof of the two continuous-observation lemmas}\label{app:continuous-proof}

We prove Lemmas~\ref{lem:factor} and~\ref{lem:frequencies} in their common setting. Let $\Omega$ be compact metric, let $\mu$ be a Borel probability completed by its null sets, and let a countable dense subgroup $G\subset\R^r$ act by measure-preserving transformations. Write $\mathcal F$ for the completed event collection and $H=L^2(\Omega,\mathcal F,\mu)$. The operators $T_g$ act by composing an observation with the appropriate transformation of $\Omega$. Besides being unitary, they preserve products, complex conjugation, constants and nonnegative functions. Set
\[
 H_c=\{v\in H:\norm{T_gv-v}_2\longrightarrow0
                 \text{ as }g\in G,\ g\to0\text{ in }\R^r\}.
\]
The limit concerns the expected squared change of an observation. No pointwise continuity of a coloring is assumed.

\begin{proof}[Proof of Lemmas~\ref{lem:factor} and~\ref{lem:frequencies}]
First, $H_c$ is a closed linear subspace. Linearity follows from the triangle inequality. If $v_n\in H_c$ and $v_n\to v$ in $L^2$, then
\[
 \norm{T_gv-v}_2
 \leq 2\norm{v-v_n}_2+\norm{T_gv_n-v_n}_2.
\]
Choose $n$ first and then make $g$ small. This proves that $v\in H_c$.

Its bounded elements also form an algebra. Here bounded means bounded outside a probability-null set, and $\norm{v}_\infty$ is the least such bound. Constants and complex conjugates belong to $H_c$. For bounded $v,w\in H_c$, preservation of products and of these bounds gives
\[
 \begin{split}
 \norm{T_g(vw)-vw}_2
 &\leq \norm{v}_\infty\norm{T_gw-w}_2
       +\norm{w}_\infty\norm{T_gv-v}_2
       \longrightarrow0.
 \end{split}
\]

Bounded elements suffice to recover the whole subspace. For $M>0$ let
\[
 \theta_M(z)=
 \begin{cases}
 z,&|z|\leq M,\\
 Mz/|z|,&|z|>M.
 \end{cases}
\]
This clips a complex value to the disk of radius $M$. As the nearest-point map onto a closed convex disk, it is $1$-Lipschitz. Composition with $\theta_M$ commutes with $T_g$, and consequently
\[
 \norm{T_g\theta_M(v)-\theta_M(v)}_2
 \leq\norm{T_gv-v}_2.
\]
Thus $v\in H_c$ implies $\theta_M(v)\in H_c$. Moreover,
\[
 |v-\theta_M(v)|^2
 \leq |v|^2\one_{\{|v|>M\}},
\]
so dominated convergence gives $\theta_M(v)\to v$ in $L^2$.

We now identify the event information represented by this subspace. Let $\mathcal A$ be the collection of all everywhere bounded, $\mathcal F$-measurable real functions whose $L^2$ classes lie in $H_c$. Including every bounded measurable representative of each class makes $\mathcal A$ a pointwise algebra, rather than merely an algebra of equivalence classes. Define $\mathcal F_c$ to be the completion by ambient null sets of
\[
 \sigma(\mathcal A)
 =\sigma\bigl(\{v^{-1}(B):v\in\mathcal A,
                              \ B\subseteq\R\text{ Borel}\}\bigr).
\]
These are the events determined by the bounded continuous observations. The inclusion of null-set modifications makes this definition independent of the representatives chosen.

The needed passage from an algebra to all its measurable functions is the \emph{functional monotone class theorem}: if an algebra $\mathcal A$ of bounded real functions contains the constant $1$ and is contained in a real vector space $\mathcal M$ of bounded functions closed under pointwise increasing limits of uniformly bounded sequences, then $\mathcal M$ contains every bounded $\sigma(\mathcal A)$-measurable real function. Apply it here with $\mathcal M=\mathcal A$. If $v_n\in\mathcal A$ increase pointwise and are uniformly bounded, their limit $v$ is measurable and $v_n\to v$ in $L^2$ by dominated convergence on the probability space. Closedness of $H_c$ gives $v\in\mathcal A$. The theorem therefore shows that every bounded $\sigma(\mathcal A)$-measurable function belongs to $H_c$. Completion changes representatives only on null sets. Taking real and imaginary parts, and then bounded truncations, yields
\[
 L^2(\mathcal F_c)\subseteq H_c.
\]
Conversely, the real and imaginary parts of each $\theta_M(v)$, for $v\in H_c$, belong to $\mathcal A$ after choosing bounded representatives. Hence $\theta_M(v)$ is $\mathcal F_c$-measurable. Its $L^2$ limit is also $\mathcal F_c$-measurable up to a null set: one may take an almost-everywhere convergent subsequence and use completion. We have proved
\begin{equation}\label{eq:continuous-information}
 H_c=L^2(\mathcal F_c).
\end{equation}
By the conditional-expectation result, Theorem~\ref{thm:conditional-expectation}, the orthogonal projection $P$ onto this space is
\[
 Pv=\mathbb E[v\mid\mathcal F_c].
\]
It therefore preserves nonnegativity and the constant $1$, as asserted in Lemma~\ref{lem:factor}.

Next we extend the parameters on $H_c$. Commutativity and unitarity show that $T_gH_c=H_c$ for every $g\in G$; the reverse inclusion follows by using $-g$. For $x\in\R^r$, choose $g_n\in G$ with $g_n\to x$. If $v\in H_c$, then
\[
 \norm{T_{g_n}v-T_{g_m}v}_2
 =\norm{T_{g_n-g_m}v-v}_2\longrightarrow0.
\]
Define $T_xv$ to be the limit in the closed space $H_c$. Comparing two approximating sequences by the same estimate proves that this definition is independent of the choice. Linearity, norm preservation and the group law pass to limits. The inverse is $T_{-x}$, so every $T_x$ is unitary. Continuity at zero on $G$ passes to real parameters by density; the group law then gives continuity at every $x$. Thus the extension is \emph{strongly continuous}: for each fixed $v\in H_c$, the map $x\mapsto T_xv$ is continuous in $L^2$ norm. The nonnegative cone of $L^2$ is closed, so these limiting operators preserve nonnegativity; they also fix $1$. This constructs an action on $H_c$, without asserting a pointwise action of real translations on every $\omega\in\Omega$.

It remains to identify this same subspace by frequencies. Let $\Pi$ be the spectral measure of the original $G$-representation. The spectral theorem gives
\begin{equation}\label{eq:continuous-spectral-norm}
 \norm{T_gv-v}_2^2
 =\int_{D_G}|d(g)-1|^2\,d\lambda_v(d).
\end{equation}
If $\lambda_v(D_G\setminus C_G)=0$, the integrand tends to zero as $g\to0$ at every character carrying the measure and is bounded by $4$. Dominated convergence implies $v\in H_c$.

For the converse, take $v\in H_c$ and use the real-parameter extension just obtained. The function
\[
 c(x)=\ip{v}{T_xv},\qquad x\in\R^r,
\]
is continuous. It is positive definite because
\[
 \sum_{j,\ell}\overline{\alpha_j}\alpha_\ell c(x_\ell-x_j)
 =\norm{\sum_\ell\alpha_\ell T_{x_\ell}v}_2^2\geq0.
\]
Bochner's theorem supplies a finite positive measure $\nu$ on $\R^r$ with
\[
 c(x)=\int_{\R^r}e^{i\xi\cdot x}\,d\nu(\xi).
\]
Map this measure into $D_G$ by $j$: explicitly, set
\[
 \widetilde\nu(S)=\nu(j^{-1}(S))
 \qquad(S\subseteq D_G\text{ Borel}).
\]
This is the distribution after applying the map $j$, also called its pushforward. It gives $D_G\setminus C_G$ mass zero. For $g\in G$ its Fourier coefficients satisfy
\[
 \int_{D_G}d(g)\,d\widetilde\nu(d)
 =\int_{\R^r}e^{i\xi\cdot g}\,d\nu(\xi)
 =c(g)=\int_{D_G}d(g)\,d\lambda_v(d).
\]

For completeness, these equalities determine the measures. Finite complex linear combinations of the evaluation functions $d\mapsto d(g)$ form an algebra: multiplying two evaluations adds their arguments. The algebra contains constants, is closed under complex conjugation because $\overline{d(g)}=d(-g)$, and separates distinct points of $D_G$. The complex Stone--Weierstrass theorem therefore makes it uniformly dense in the continuous functions on $D_G$. Equality of its integrals implies equality of integrals of all continuous functions, because both measures are finite. Finite Borel measures on a compact metric space are regular and are uniquely determined by those continuous tests. Hence $\widetilde\nu=\lambda_v$, and $\lambda_v$ is carried by $C_G$.

We have proved both implications in
\[
 v\in H_c\quad\Longleftrightarrow\quad
 \lambda_v(D_G\setminus C_G)=0.
\]
The right side describes the range of $\Pi(C_G)$: its complementary projection has squared norm $\lambda_v(D_G\setminus C_G)$. Thus the two orthogonal projections onto the same subspace coincide,
\begin{equation}\label{eq:continuous-projection-identification}
 P=\Pi(C_G).
\end{equation}
In particular, $v\perp H_c$ implies
\[
 \lambda_v(C_G)=\norm{\Pi(C_G)v}_2^2=0.
\]
This proves Lemma~\ref{lem:frequencies} for precisely the subspace already identified in Lemma~\ref{lem:factor}.
\end{proof}

\subsection{When another symmetry commutes with the projection}\label{app:commutation}

Let $M$ be a closed subspace of a Hilbert space, let $P_M$ be its orthogonal projection, and let $U$ be unitary. Suppose that both $U$ and $U^{-1}$ carry $M$ into $M$. This says $UM=M$. It also forces preservation of the perpendicular subspace: if $b\in M^\perp$ and $w\in M$, then
\[
 \ip{w}{Ub}=\ip{U^{-1}w}{b}=0,
\]
because $U^{-1}w\in M$. Applying the same argument to $U^{-1}$ gives $UM^\perp=M^\perp$. Therefore $U$ preserves both parts of the decomposition $v=P_Mv+(v-P_Mv)$, and its uniqueness gives
\[
 UP_M=P_MU.
\]

For the continuous subspace $H_c$, the hypothesis is easy to verify when $U$ commutes with every $T_g$: the continuity norm for $Uv$ then equals that for $v$, and the same holds for $U^{-1}$. More generally, suppose
\[
 U^{-1}T_gU=T_{Lg},
\]
where $L$ is an invertible linear map of $\R^r$ with $L(G)=G$. Then
\[
 \norm{T_gUv-Uv}_2=\norm{T_{Lg}v-v}_2\longrightarrow0
 \qquad(v\in H_c, g\to0).
\]
Using $L^{-1}$ proves the corresponding assertion for $U^{-1}$. Thus $UP=PU$ in both cases. The main proof applies this criterion to commuting translations and to rotations of coordinates. The inverse is what ensures preservation of the perpendicular subspace as well as the continuous observations themselves.

\section{Averaging the algebraic height parameter}
\label{app:height-averaging}

We prove Lemma~\ref{lem:smoothing}. The difficulty is that the vectors
$V_tv$ are initially available only at algebraic heights $t\in F$.
Ordinary convolution on $\R$ would require values at all real heights;
continuity is what the averaging argument is meant to produce. We
instead construct an operation on functions of algebraic heights,
following the invariant integral used in~\cite[Section~2.2]{OpenAI}.
First it returns a number from a scalar function. Then inner products
determine its vector output. The last step shows why this operation
selects exactly the vectors with continuous translation orbits.

\subsection{Scalar averaging by periodization}
\label{app:height-scalar}

An invariant mean on an additive group is a linear rule for averaging
all bounded complex-valued functions on that group. The rule is
positive, takes the constant function $1$ to $1$, and gives the same
answer after every translation. Abelian discrete groups admit invariant
means; see~\cite[Theorem~G.2.1]{BHV}. Choose one, denoted by $m$, on
$F/\Z$. This quotient identifies two algebraic heights when their
difference is an integer. Thus
\[
 m(1)=1,\qquad
 m(\psi(\,\cdot+a))=m(\psi)\quad(a\in F/\Z).
\]
Positivity and normalization imply
$|m(\psi)|\leq m(|\psi|)\leq\norm{\psi}_\infty$.
No measurability condition is imposed on $\psi$.

Let $h:F\to\C$ be bounded and zero outside a bounded real interval.
Add together its values at all integer translates of a height:
\[
 \mathcal P h(t+\Z)=\sum_{n\in\Z}h(t+n).
\]
This is called periodization. Choosing a different representative of
$t+\Z$ merely reindexes the sum. If $h$ vanishes outside $[-R,R]$,
only a bounded number of integers can contribute, uniformly in $t$.
Consequently $\mathcal P h$ is a bounded function on $F/\Z$, and we may
define
\begin{equation}\label{eq:height-scalar-definition}
 I_{\C}(h)=m(\mathcal P h).
\end{equation}
This functional is linear and positive. Moreover,
$\mathcal P(h(\,\cdot+s))(t+\Z)=\mathcal P h(t+s+\Z)$, so it is
translation-invariant for every $s\in F$.

The integer sum supplies the normalization needed for local averages:
$\mathcal P\one_{[0,1)\cap F}=1$, and hence
$I_{\C}(\one_{[0,1)\cap F})=1$. A normalized invariant mean on $F$
itself would give zero to every boundedly supported function: any
bounded interval has arbitrarily many disjoint translates, so positivity
bounds its mean by $1/N$ for every $N$. The quotient construction
retains the size of a bounded interval instead of erasing it.

We next check that this new operation has the usual value on continuous
functions. A continuous function on the circle can be viewed as a
continuous $1$-periodic function on $\R$. Its basic Fourier functions are
\[
 e_n(t)=e^{2\pi int},\qquad n\in\Z.
\]
For $n\ne0$, translation by the rational height $1/(2n)\in F$ changes
$e_n$ to $-e_n$. Translation invariance therefore gives
$m(e_n|_{F/\Z})=0$. For $n=0$ the mean is $1$.
These are also the values of $\int_0^1e_n(t)\,dt$.
Linearity gives agreement for every finite linear combination of the
$e_n$. The uniform trigonometric approximation theorem states that
every continuous $1$-periodic function is a uniform limit of such
combinations. Since both $m$ and ordinary integration are bounded by
the supremum norm, it follows that
\[
 m(f|_{F/\Z})=\int_0^1 f(t)\,dt
 \quad\text{for every continuous $1$-periodic }f.
\]

If $\varphi\in C_c(\R)$, then
$\mathcal P\varphi(t)=\sum_n\varphi(t+n)$ is continuous and
$1$-periodic. Only finitely many summands are nonzero on $[0,1]$.
Thus the preceding identity and a finite interchange of sum and
ordinary integral give
\begin{equation}\label{eq:height-agreement}
 \begin{aligned}
 I_{\C}(\varphi|_F)
 &=\int_0^1\sum_{n\in\Z}\varphi(t+n)\,dt\\
 &=\sum_{n\in\Z}\int_n^{n+1}\varphi(x)\,dx
 =\int_{\R}\varphi(x)\,dx.
 \end{aligned}
\end{equation}
Restrictions to $F$ will be understood below. The constant function
$1$ on all of $F$ is outside the domain of $I_{\C}$; normalization here
means that a compactly supported continuous kernel of integral one has
$I_{\C}(\varphi)=1$.

We will also need an estimate for functions without a continuous
extension. Positivity implies
\begin{equation}\label{eq:height-absolute-bound}
 |I_{\C}(h)|\leq I_{\C}(|h|).
\end{equation}
Indeed, multiply $h$ by a complex number of modulus one which makes
its integral real and nonnegative, and compare its real part pointwise
with $|h|$. In particular, for $\varphi\geq0$ in $C_c(\R)$ and a bounded
function $h$ on $F$,
\[
 |I_{\C}(\varphi h)|
 \leq I_{\C}(\varphi)
       \sup_{t\in F\cap\supp\varphi}|h(t)|.
\]
For a kernel of integral one this is
\eqref{eq:height-support-bound}; the case $\varphi=0$ is immediate.
All these conclusions use positivity, finite sums, translation
invariance, and uniform approximation. We do not assume countable
additivity of $m$, and apply no interchange or convergence theorem for
countably additive measures to it.

\subsection{Vector averages determined by inner products}
\label{app:height-vector}

Let $W:F\to H$ be bounded in norm and vanish outside a bounded real
interval. In the coloring application, each $W(t)$ is a function of a
random coloring; averaging over $t$ must leave such a function as the
output. We determine that output by its inner product with every test
vector $w\in H$.

Recall the Hilbert-space Riesz representation theorem in the convention
used here: every bounded conjugate-linear functional $\ell$ on $H$ has
the form $\ell(w)=\ip{w}{y}$ for a unique $y\in H$, with
$\norm{y}=\sup_{\norm{w}\leq1}|\ell(w)|$.
For the present $W$, define
\[
 \ell_W(w)=I_{\C}\bigl(t\mapsto\ip{w}{W(t)}\bigr).
\]
The scalar function has bounded support, so it is in the domain of
$I_{\C}$. It is conjugate-linear in $w$ because our inner product is
linear in its second argument. By \eqref{eq:height-absolute-bound} and
Cauchy--Schwarz,
\[
 |\ell_W(w)|
 \leq\norm{w}\,I_{\C}\bigl(t\mapsto\norm{W(t)}\bigr).
\]
Riesz representation therefore gives a unique vector $I_H(W)$ with
\begin{equation}\label{eq:height-vector-definition}
 \ip{w}{I_H(W)}
   =I_{\C}\bigl(t\mapsto\ip{w}{W(t)}\bigr)
       \quad(w\in H),
 \qquad
 \norm{I_H(W)}\leq I_{\C}(\norm{W(\cdot)}).
\end{equation}
Thus the length of a vector average is at most the average of the
lengths. Equality against all test vectors also proves linearity in
$W$ and the useful identity $I_H(t\mapsto h(t)v)=I_{\C}(h)v$.

Let $L:H\to H$ be any bounded linear operator. Its adjoint is defined
by $\ip{w}{Ly}=\ip{L^*w}{y}$. Hence
\[
 \begin{aligned}
 \ip{w}{L I_H(W)}
 &=\ip{L^*w}{I_H(W)}\\
 &=I_{\C}\bigl(t\mapsto\ip{L^*w}{W(t)}\bigr)\\
 &=I_{\C}\bigl(t\mapsto\ip{w}{L W(t)}\bigr)
 =\ip{w}{I_H(LW)}.
 \end{aligned}
\]
It follows that $L I_H(W)=I_H(LW)$. Likewise, scalar translation
invariance, tested against every $w$, gives
$I_H(W(\,\cdot+s))=I_H(W)$ for $s\in F$.
These are the two rules that allow us to move a translation operator
through the average and then change the height parameter.

\subsection{Continuous outputs and the shrinking-kernel limit}
\label{app:height-limit}

First, $H_v$ is a closed linear subspace. To see closedness, if
$v_n\in H_v$ and $v_n\to v$ in norm, then
\[
 \norm{V_hv-v}
 \leq2\norm{v-v_n}+\norm{V_hv_n-v_n}.
\]
Choose $n$ first and then take $h\in F$ sufficiently close to zero.
This proves that $v\in H_v$ and justifies its orthogonal projection $Q$.
Commutativity gives
$\norm{V_hV_tw-V_tw}=\norm{V_hw-w}$, so $V_t$ preserves $H_v$.
The same holds for its inverse $V_{-t}$; consequently it also preserves
$H_v^\perp$.

For $\varphi\in C_c(\R)$ put
\[
 S_\varphi v=I_H(t\mapsto\varphi(t)V_tv).
\]
Unitarity and \eqref{eq:height-vector-definition} give
\begin{equation}\label{eq:height-operator-bound}
 \norm{S_\varphi v}
 \leq\norm{v}\,I_{\C}(|\varphi|)
 =\norm{v}\int_{\R}|\varphi(t)|\,dt.
\end{equation}
To examine its translated orbit, use the two rules just proved:
\[
 \begin{aligned}
 V_hS_\varphi v
 &=I_H(t\mapsto\varphi(t)V_{t+h}v)\\
 &=I_H(t\mapsto\varphi(t-h)V_tv).
 \end{aligned}
\]
The difference from $S_\varphi v$ therefore satisfies
\begin{equation}\label{eq:height-continuity-bound}
 \begin{aligned}
 \norm{V_hS_\varphi v-S_\varphi v}
 &\leq\norm{v}\,I_{\C}
       \bigl(t\mapsto|\varphi(t-h)-\varphi(t)|\bigr)\\
 &=\norm{v}\int_{\R}|\varphi(t-h)-\varphi(t)|\,dt
 \longrightarrow0.
 \end{aligned}
\end{equation}
For small $h$, the integrands vanish outside one fixed bounded
interval and converge uniformly to zero, by uniform continuity of
$\varphi$. Their ordinary integrals consequently tend to zero.
This proves $S_\varphi v\in H_v$ for every $v\in H$, without requiring
continuity of the original orbit $t\mapsto V_tv$.

Now let $v\perp H_v$. For $w\in H_v$ and $t\in F$,
\[
 \ip{w}{V_tv}=\ip{V_{-t}w}{v}=0.
\]
Using these scalar tests in \eqref{eq:height-vector-definition} shows
$S_\varphi v\perp H_v$. But \eqref{eq:height-continuity-bound} already
put this vector in $H_v$. Thus
\begin{equation}\label{eq:height-kills-complement}
 S_\varphi v=0\qquad(v\in H_v^\perp).
\end{equation}

Finally, let $v\in H_v$ and take the kernels in
Lemma~\ref{lem:smoothing}. Since
$I_{\C}(\varphi_\varepsilon)=1$, the vector norm bound and positivity
give
\begin{equation}\label{eq:height-approximation-bound}
 \begin{aligned}
 \norm{S_{\varphi_\varepsilon}v-v}
 &=\norm{I_H(t\mapsto
       \varphi_\varepsilon(t)(V_tv-v))}\\
 &\leq I_{\C}\bigl(t\mapsto
       \varphi_\varepsilon(t)\norm{V_tv-v}\bigr)\\
 &\leq\sup_{\substack{t\in F\\|t|\leq\delta_\varepsilon}}
       \norm{V_tv-v}
 \longrightarrow0.
 \end{aligned}
\end{equation}
The last step is exactly the definition of $H_v$, which controls every
sufficiently small algebraic height. For an arbitrary vector, decompose
$v=Qv+(v-Qv)$. Equation~\eqref{eq:height-kills-complement} removes the
second component, and \eqref{eq:height-approximation-bound} sends the
average of the first to $Qv$. This proves \eqref{eq:smoothing}.
Equation~\eqref{eq:height-operator-bound} also gives
$\norm{S_{\varphi_\varepsilon}}\leq1$. Although the vector averages
were defined by scalar tests, the conclusion is convergence in Hilbert
norm for each vector, as the displayed estimate shows.

In the probability setting of the lemma, apply
Lemma~\ref{lem:factor} to the countable dense group $F\subset\R$ and
its measure-preserving height translations. It identifies $H_v$ with
the functions measurable for the corresponding retained collection of
events, and $Q$ with conditional expectation onto that collection.
Consequently $Q$ preserves nonnegative functions and fixes $1$.
This last assertion uses the probability action; it is not an order
property asserted for an arbitrary Hilbert space.

\section*{Statement of AI use}

We started our work on the chromatic number of $\R^n$ in late August.
On the 7'th of October after the OpenAI release we used ChatGPT~6 Astra
to obtain the main ideas from~\cite{OpenAI} and were able to produce the
$(2+o(1))^n$ lower bound along with the result we present here. After the
work of Davies~\cite{Davies} was released we decided to publish this
remaining result that closes the gap between $n=2$ from~\cite{OpenAI} and
$n\ge4$ from~\cite{Davies}. We also used ChatGPT~6 Astra for editing the
paper and proofreading.

The authors are responsible for the correctness of the proofs presented in the paper.

\phantomsection

\bibliographystyle{abbrv}
\bibliography{main}

\end{document}